\documentclass[12pt,a4paper]{amsart}
\pdfoutput=1

\def\thispapertitle{Discrete multiple orthogonal polynomials on~shifted lattices}

\usepackage{array}
\usepackage{arydshln}
\usepackage[
british]{babel}
\usepackage[utf8]{inputenc}
\usepackage[
T1]{fontenc}
\usepackage[margin=2cm]{geometry}
\usepackage{graphicx}
\DeclareGraphicsExtensions{.pdf}
\usepackage[selected=true,shrink=10,stretch=40,verbose=true,tracking=alltext,letterspace=0
]{microtype}
\usepackage{amssymb}
\usepackage{paralist}

\usepackage[colorlinks=false,
pdfborderstyle={/S/D/D[2 1]/W .5},
linktocpage=true,unicode=true,pdftex, pdfversion=1.5, pdftitle={\thispapertitle},
pdfauthor={Alexander Dyachenko, Vladimir Lysov},
pdfsubject={MSC(2010) 33C45 (42C05)},
pdfkeywords={Classical orthogonal polynomials of descrete variable, Multiple orthogonality,
    Pearson equation, Rodrigues formula}, pdfpagemode=UseOutlines,pdflang=en-GB,
]{hyperref}

\makeatletter
\renewcommand\subsubsection{\@startsection{subsubsection}{3}%
    \z@{.5\linespacing\@plus.7\linespacing}{-.5em}%
    {\normalfont\bfseries}}%
\renewcommand\@settitle{\begin{center}%
        \baselineskip14\p@\relax%
        \Large\scshape%
        \uppercasenonmath\@title%
        \@title%
    \end{center}}
\makeatother
\newtheorem{theorem}{Theorem}
\newtheorem{lemma}[theorem]{Lemma}
\newtheorem{prop}[theorem]{Proposition}

\newcommand{\coloneqq}{\mathrel{\mathop:}=}

\newcommand{\supp}{\operatorname{supp}}
\renewcommand{\le}{\leqslant}
\renewcommand{\ge}{\geqslant}

\begin{document}

\title{\thispapertitle}

\author{Alexander Dyachenko}
\thanks{The first author is supported by the research fellowship DY 133/1-1 of the German
    Research Foundation~(DFG)}
\address[A.~Dyachenko]{UCL Department of Mathematics, Gower St, London WC1E 6BT, UK}
\email[A.~Dyachenko]{a.dyachenko@ucl.ac.uk, diachenko@sfedu.ru}

\author{Vladimir Lysov}
\thanks{The second author is supported by Russian Science Foundation, grant No. 19-71-30004}
\address[V.~Lysov]{Keldysh Institute of Applied Mathematics (RAS), Miusskaya sq. 4, 125047
    Moscow, Russia}
\address[V.~Lysov]{ Moscow Institute of Physics and Technology (State University), Dolgoprudny,
    Moscow region, Russia} \email[V.~Lysov]{v.g.lysov@gmail.com}

\date{\today} 

\begin{abstract}
    We introduce a new class of polynomials of multiple orthogonality with respect to the
    product of~$r$ classical discrete weights on integer lattices with noninteger shifts. We
    give explicit representations in the form of the Rodrigues formulas. The case of two weights
    is described in detail.
\end{abstract}

\maketitle   
\section{Introduction}
Orthogonal polynomials of discrete variable have numerous applications. One the most well-known
applications is related to the theory of representations of the 3D rotation group~\cite{GMS}: it
is based on a deep connection between polynomials orthogonal with respect to classical discrete
and continuous measures. The similarity between these polynomials is nicely exposed in the
monograph~\cite{NSU}.

The current paper focuses on multiple orthogonal polynomials induced by a system of discrete
measures that generalize the classical ones. Our study makes use of the analogy with the
continuous case. For that reason, we first remind some properties of classical polynomials
orthogonal with respect to continuous measures.

\subsection{Classical orthogonal polynomials}

Classical orthogonal polynomials are the polynomials by Hermite, Laguerre, and Jacobi (the last
also include polynomials by Chebyshev and Gegenbauer as special cases). They constitute a simple
yet quite important class of special functions. These polynomials arise in a wide range of
problems in quantum mechanics, mathematical physics, numerical analysis and approximation
theory. Their classification and derivation of their main properties base upon the Pearson
equation for the weight of orthogonality. According to the definition in~\cite{NikUv}, classical
orthogonal polynomials are the polynomials~$P_n$ of degree~$n$ that satisfy the orthogonality
relations with respect to the continuous weight~$\rho$ on the interval~$S_\rho\subset\mathbb R$:
\begin{equation}\label{eq:orthog_cont}
    \int_{S_\rho} P_n(x) x^k \rho(x)\, dx = 0,
    \quad
    k=0,1,\dots,n-1,
\end{equation}
where~$\rho$ satisfies the Pearson differential equation:
\begin{equation}\label{eq:Pearson_cont}
\frac{d}{dx} \big(\sigma (x) \rho(x)\big) = \tau(x)\rho(x)
\end{equation}
with polynomials~$\sigma,\tau$ of degrees~$\deg \sigma\le 2$ and~$\deg\tau=1$. Up to shifts and
rescaling, there are precisely three%
\footnote{Some authors additionally include in this list the Bessel polynomials which are
    orthogonal on the cardioid in~$\mathbb{C}$ with respect to~$e^{-\frac 1x}$. For our goal, it
    is enough only to consider orthogonality on subsets of the real line~$\mathbb{R}$.} %
types of classical orthogonal polynomials, see Table~\ref{tab:classical_OP}.

\begin{table}
    \setlength{\extrarowheight}{10pt}
    \centering
    \caption{Classical orthogonal polynomials}
    \label{tab:classical_OP}
    \begin{tabular}{|c|c;{.4pt/1pt}c;{.4pt/1pt}c|}
      \hline
      $P_n$
        & $H_n$, Hermite
          & $L^\alpha_n$, Laguerre
            & $P^{\alpha,\beta}_n$, Jacobi
      \\[3pt] \hline
      $\rho(x)$
        & $\displaystyle e^{-x^2}$
          & $\displaystyle x^\alpha e^{-x}$
            & $\displaystyle (1+x)^\alpha (1-x)^\beta$
      \\[3pt] \hdashline[.4pt/1pt]
      $\displaystyle S_\rho$
        & $(-\infty,\infty)$
          & $(0,\infty)$
            & $(-1,1)$
      \\[3pt] \hdashline[.4pt/1pt]
      $\displaystyle \sigma(x)$
        & $1$
          & $x$
            & $x^2-1$
      \\[3pt] \hdashline[.4pt/1pt]
      $\displaystyle \tau(x)$
        & $-2x$
          & $-x+\alpha+1$
            & $(\alpha+\beta+2)x-\alpha+\beta$
      \\[3pt] \hdashline[.4pt/1pt]
      C1
        & ---
          & $\alpha>-1$
            & $\alpha>-1$, $\beta>-1$
      \\[3pt] \hline
    \end{tabular}
\end{table}

The polynomials~$P_n$ may be written explicitly through the Rodrigues formula:
\begin{equation}\label{eq:Rodrigues_cont}
    P_n(x) \rho(x) = C_n \frac{d^n}{dx^n} \big( \sigma^n (x) \rho(x) \big)
    ,
\end{equation}
where~$C_n\ne 0$ stands for a normalization constant. Condition~C1 in
Table~\ref{tab:classical_OP} yields positivity and integrability of the weight~$\rho$. All zeros
of~$P_n$ are simple and belong to~$S_\rho$.

\subsection{Multiple orthogonal polynomials}

Some applications related to rational approximations, random matrices, Diophantine
approximations give rise to multiple orthogonal polynomials, see
e.g.~\cite{AptKui11,Kuijlaars10,MukhVar07,VanAsshe99,Sor2002}. These polynomials are determined
by orthogonality relations with respect to more than one weight.

Let~$\rho_1,\dots,\rho_r$ be the weight functions determined on the
intervals~$S_{\rho_1},\dots,S_{\rho_r}$ of the real line. A polynomial~$P_n$ is called a (type
II) multiple orthogonal polynomial if~$P_n\not\equiv 0$,~$\deg P_n\le rn$ and the orthogonality
relations:
\begin{equation}\label{eq:mult_orthog_cont}
    \int_{S_{\rho_j}} P_n(x) x^k\rho_j(x)\,dx =0,
    \quad
    k=0,\dots,n-1,
    \quad
    j=1,\dots,r
\end{equation}
are satisfied. The case~$r=1$ corresponds to the standard orthogonality~\eqref{eq:orthog_cont}.
Two ways to determine classical multiple orthogonal polynomials are currently in use. The first
one~\cite{ABV} is when the intervals~$S_{\rho_j}$ coincide, i.e.~$S_{\rho_j}=\cdots=S_{\rho_j}$,
and the weights~$\rho_j$ satisfy the Pearson equation~\eqref{eq:Pearson_cont} with
distinct~$\tau^{(j)}$. In this case, the restrictions for the degrees of~$\sigma$ and~$\tau^{(j)}$
remain the same: $\deg\sigma\le 2$ and~$\deg\tau^{(j)}=1$

\begin{table}[b]
    \setlength{\extrarowheight}{10pt}
    \centering
    \caption{Classical multiple orthogonal polynomials}
    \label{tab:cont_MOP}
    \begin{tabular}{|c|c;{.4pt/1pt}c;{.4pt/1pt}c|}
      \hline
      $P_n$
        & Laguerre-Hermite
          & Jacobi-Laguerre
            & Angelesco-Jacobi
      \\[3pt] \hline
      $\rho(x)$
        & $\displaystyle x^\beta e^{-x^2}$
          & $\displaystyle (a+x)^\alpha x^\beta e^{-x}$
            & $\displaystyle (a+x)^\alpha x^\beta (1-x)^\gamma$
      \\[3pt] \hdashline[.4pt/1pt]
      $\rho_j(x)$
        & $\displaystyle |x|^\beta e^{-x^2}$
          & $\displaystyle (a+x)^\alpha |x|^\beta e^{-x}$
            & $\displaystyle (a+x)^\alpha |x|^\beta (1-x)^\gamma$
      \\[3pt] \hdashline[.4pt/1pt]
      $\displaystyle S_{\rho_1}, S_{\rho_2}$
        & $(-\infty,0)$, $(0,\infty)$
          & $(-a,0)$, $(0,\infty)$
            & $(-a,0)$, $(0,1)$
      \\[3pt] \hdashline[.4pt/1pt]
      $\displaystyle \sigma(x)$
        & $x$
          & $(x+a)x$
            & $(x+a)x(x-1)$
      \\[3pt] \hdashline[.4pt/1pt]
      $\displaystyle \tau(x)$
        & $-2x^2+\beta+1$
          &
            \begin{tabular}{l}
              $-x^2$
              \\[-5pt]
              $+(\alpha+\beta+2-a)x$
              \\[-5pt]
              $+(\beta+1)a$
            \end{tabular}
            &
              \begin{tabular}{l}
                $(\alpha+\beta+\gamma+3)x^2$
                \\[-5pt]
                ${}+\big((\beta+\gamma+2)a$
                \\[-5pt]
                \hspace{1em}${}-(\alpha+\beta+2)\big)x-(\beta+1)a$
                \\[5pt]
              \end{tabular}
      \\[3pt] \hdashline[.4pt/1pt]
      MC1
        & $\beta>-1$
          &
            \begin{tabular}{l}
              $\alpha>-1$, $\beta>-1$,
              \\[-5pt]
              $a>0$
            \end{tabular}
          &
            \begin{tabular}{l}
              $\alpha>-1$, $\beta>-1$,
              \\[-5pt]
              $\gamma>-1$, $a>0$
            \end{tabular}
      \\[3pt]\hline
    \end{tabular}
\end{table}
There is also another way~\cite{Aptekarev98,AMR,VC}, which we consider in the particular
case~$r=2$. Here, the weight functions~$\rho_1,\rho_2$ are proportional to restrictions of a
certain analytic function~$\rho$ to disjoint intervals~$S_{\rho_1},S_{\rho_2}$, that
is:~$\rho_j\coloneqq c_j\rho|_{S_{\rho_j}}$ for a constant~$c_j$. The function~$\rho$ is
additionally assumed to satisfy the Pearson equation~\eqref{eq:Pearson_cont} with the
polynomials~$\sigma$ and~$\tau$ of degrees~$\deg \sigma\le 3$ and~$\tau=2$. This construction
yields the Angelesco-Jacobi~\cite{Angelesco,Kaljagin}, Jacobi-Laguerre~\cite{Sor1984} and
Laguerre-Hermite~\cite{Sor1986} polynomials. The basic characteristics of these polynomials are
stated in Table~\ref{tab:cont_MOP}. Condition~MC1 (Multiple Continuous) provides positivity and
integrability of the weights, as well as that~$S_{\rho_1} \cap S_{\rho_2} = \varnothing$. The
orthogonality~\eqref{eq:mult_orthog_cont} immediately implies that the polynomial~$P_n$ has~$n$
simple zeros in~$S_{\rho_j}$. One of the most remarkable properties of this construction is that
the resulting polynomials also satisfy the Rodrigues formula~\eqref{eq:Rodrigues_cont}.

\subsection{Classical orthogonal polynomials of discrete variable}

Up until this point we discussed the continuous case. Now, let us consider orthogonality with
respect to discrete measures. Let~$S_\rho$ be a discrete subset of~$\mathbb{R}$,
e.g.~$S_\rho=\mathbb{Z}_+\coloneqq \{0,1,\dots\}$ or~$S_\rho=\{0,1,\dots, N\}$ for some
integer~$N\ge 0$. The function~$\rho$ is supposed to be nonnegative on~$S_\rho$. Put the
charge~$\rho(x)$ in each point~$x\in S_\rho$ and consider the discrete measure:
\begin{equation}\label{eq:discrete_measure_def}
    \mu(y) = \sum_{x\in S_\rho} \rho(x) \delta(y-x),
\end{equation}
where~$\delta$ is the Dirac delta. Then the function~$\rho$ is called the discrete weight of the
measure~$\mu$.

Consider a sequence of polynomials~$P_n$ of degree~$n$, which are orthogonal with respect to the
measure~$\mu$:
\begin{equation}\label{eq:orthog_discr}
    \int P_n(x) x^k d\mu(x) = \sum_{x\in S_\rho} P_n(x)x^k\rho(x)=0,
    \quad
    k=0,1,\dots,n-1.
\end{equation}
Denote by~$\Delta f$ and~$\nabla f$ the left and the right finite differences of~$f$,
respectively:
\[
    \Delta f(x)=f(x+1)-f(x)
    \quad\text{and}\quad
    \nabla f(x)=f(x)-f(x-1)
    .
\]
As above, given two polynomials~$\sigma$ and~$\tau$ such that~$\deg\sigma\le 2$ and~$\deg\tau=1$
one may also consider the (difference) Pearson equation:
\begin{equation}\label{eq:discr_Pearson}
    \Delta \big( \sigma(x) \rho(x) \big)
    = \tau(x) \rho(x)
    .
\end{equation}
Up to trivial transformations, its regular solutions determine precisely four (see~\cite{NSU})
classes of classical orthogonal polynomials of discrete variable: the Charlier and Mexner
polynomials are orthogonal on~$\mathbb{Z}_+$, while the Kravchuk (Krawtchouk) and Hahn
polynomials are orthogonal on~$\{0,1,\dots,N\}$. Polynomials of these classes satisfy the
Rodrigues formula:
\begin{equation}\label{eq:discr_Rodrigues}
    P_n(x) \rho (x) = C_n \nabla^n\rho_n(x),
\end{equation}
where~$C_n$ is a normalizing constant and~$\rho_n$ is a shifted weight, see
Table~\ref{tab:Dclassical}. In this table, the Hahn weight is expressed via the so-called
Pochhammer symbol:
\[
    (\alpha)_x
    \coloneqq \frac{\Gamma(x+\alpha)}{\Gamma(\alpha)}
\]
and, in particular, $(\alpha)_x=(\alpha)(\alpha+1)\cdots(\alpha+x-1)$ for~$x\in\mathbb{Z}_+$.
The Kravchuk and Hahn polynomials additionally require that~$n\le N$.

\begin{table}[h]
    \setlength{\extrarowheight}{10pt}
    \centering
    \caption{Classical orthogonal polynomials of a discrete variable}
    \label{tab:Dclassical}
    \begin{tabular}{|c|c;{.4pt/1pt}c;{.4pt/1pt}c;{.4pt/1pt}c|}
      \hline
      $P_n$
      & Charlier
      & Meixner
      & Kravchuk
      & Hahn\\[3pt]
      \hline
      $\rho(x)$
      & $\displaystyle\frac{b^x}{\Gamma(x+1)}$
      & $\displaystyle\frac{b^x\Gamma(x+\alpha)}{\Gamma(x+1)}$
      & \!\!$\frac{\displaystyle b^x}{\Gamma(x+1)\Gamma(N-x+1)}$\!\!
      & $\displaystyle
        \frac{(\alpha)_x(\beta)_{N-x}}
        {\Gamma(x+1)\Gamma(N-x+1)}$\\[10pt]
      \hdashline[.4pt/1pt]
      $S_\rho$
      & $\mathbb{Z}_+$
      & $\mathbb{Z}_+$
      & $\{0,1,\dots,N\}$
      & $\{0,1,\dots,N\}$\\[3pt]
      \hdashline[.4pt/1pt]
      $\rho_m(x)$
      & $\displaystyle\frac{b^x}{\Gamma(x+1)}$
      & \!$\displaystyle\frac{b^x\Gamma(x+\alpha+m)}{\Gamma(x+1)}$\!
      & \!$\frac{\displaystyle b^x}{\Gamma(x+1)\Gamma(N-m-x+1)}$\!
      & \!$\displaystyle
        \frac{(\alpha+m)_x(\beta+m)_{N-m-x}}{\Gamma(x+1)\Gamma(N -m-x+1)}$\!\\[10pt]
      \hdashline[.4pt/1pt]
      $\displaystyle\frac{\rho_{m+1}(x)}{\rho_m(x)}$
      & $1$
      & $x+\alpha+m$
      & $N-m-x$
      & $\displaystyle\frac{(x+\alpha+m)(N -m-x)}{(\alpha+m)(\beta+m)}$\!\!\\[10pt]
      \hdashline[.4pt/1pt]
      \!$\displaystyle \frac{\rho_{m+1}(x-1)}{\rho_m(x)}$\!
      & $\displaystyle \frac xb$
      & $\displaystyle \frac xb$
      & $\displaystyle \frac xb$
      & $\displaystyle\frac{x(N-x+\beta)}{(\alpha+m)(\beta+m)}$\\[10pt]
      \hdashline[.4pt/1pt]
      D1
      & $b>0$
      & $\alpha\!>\!0,\ 0\!<\!b\!<\!1$
      & $b>0$
      & \!\!\!\!$\alpha,\beta>0$ \ or \ $\alpha,\beta<1-N$\!\!\!\!\!\!\\[3pt]
      \hdashline[.4pt/1pt]
      D2
      & $b\ne 0$
      & \begin{tabular}{c}
          \\[-2.5em]
          $0<|b|<1$,\\[-7pt]
          $\alpha\notin-S_\rho$\\[2pt]
        \end{tabular}
      & $b\ne 0$, $b\ne -1$
      & \!\begin{tabular}{c}
          \\[-2.5em]
          $\alpha,\beta\notin\{0,-1,\dots,1-N\}$,\\[-7pt]
          $-\alpha-\beta\notin\{0,1,\dots,2N-2\}$\!\\[2pt]
        \end{tabular}\!\!\!
      \\[3pt]
      \hline
    \end{tabular}
\end{table}

Condition~D1 in Table~\ref{tab:Dclassical} yields positivity%
\footnote{For the Hahn case, the inequality~$\alpha,\beta<N-1$ in Condition~D1 induces that the
    weight keeps the same sign~$(-1)^{N}$ on the whole support~$S_\rho$. In general, our
    Table~\ref{tab:Dclassical} and Condition~D1 may be compared to~\cite{DLMF}.}
and finiteness of all moments for the measure~$\mu$ from~\eqref{eq:discrete_measure_def}. The
orthogonality relations immediately imply that all zeros of~$P_n$ are simple and lie in the
convex hull of~$S_\rho$. Between two consecutive lattice points there is at most one zero
of~$P_n$.

\subsection{Polynomials of multiple discrete orthogonality}

The paper~\cite{ACV} implements the first of the schemes described above for constructing
classical multiple orthogonal polynomials (``different weights on the same support''). Analogous
to the continuous case,~\cite{ACV} deals with discrete measures with a common support, that
satisfy the difference Pearson equation~\eqref{eq:discr_Pearson} with polynomials~$\sigma$
and~$\tau^{(j)}$ of degrees~$\deg\sigma\le 2$ and~$\deg\tau^{(j)}=1$.

To implement the second scheme (``one weight on different supports'') in the discrete case, one
needs to study the lattices that are shifted by noninteger numbers. This observation seems to be
first made by Sorokin in~\cite{Sor2010}, where he studied various generalizations of the Meixner
polynomials. Our paper aims at implementing this approach in a more general setting. We also
highlight the recent publications~\cite{DomMar14,FilVanAs18,SmVanAs12} devoted to orthogonal
polynomials (i.e. $r=1$) with respect to products of classical weights on a union of lattices
with shifts.

We begin with the following observation: the Rodrigues formula~\eqref{eq:discr_Rodrigues}
determines an orthogonal polynomial~$P_n$ of degree~$n$ under rather general conditions when the
parameters~$\alpha,\beta, b$ have complex values.

\begin{prop}\label{th:2}
    In Table~\ref{tab:Dclassical}, if the parameters of the function~$\rho$ satisfy
    Condition~D2, then the polynomial~$P_n$ determined by~\eqref{eq:discr_Rodrigues} is of
    degree~$n$ and satisfies the orthogonality conditions~\eqref{eq:orthog_discr} with respect
    to (generally speaking, complex-valued) measure~$\mu$ from~\eqref{eq:discrete_measure_def}.
\end{prop}

The proof is given in the next section. To formulate a similar proposition regarding multiple
orthogonality, we need an analogue to the condition~D2, which will be denoted by~MD2 (Multiple
Discrete).

Suppose that~$\rho^{(1)},\dots,\rho^{(r)}$ is a set of classical discrete weights whose
parameters satisfy the conditions~D2. Let us introduce more specific notation. Denote by~$r_C$,
$r_M$, $r_K$ and~$r_H$, respectively, the number of the Charlier, Meixner, Kravchuk and Hahn
weights among~$\rho^{(1)},\dots,\rho^{(r)}$. In particular, $r_C + r_M + r_K + r_H = r$. Each of
the weights corresponds to its own parameters. To unify, we relate each weight~$\rho^{(j)}$ to
the quadruple of its parameters~$(\alpha_j,\beta_j,b_j,N_j)$ according to the following rules:
\begin{itemize}
\item The Hahn weight corresponds to~$b_j=1$, other weights correspond to~$\beta_j=0$.
\item For the Charlier and Kravchuk weights,~$\alpha_j=0$, while the Charlier and Meixner
    weights have~$N_j=+\infty$.
\end{itemize}
Put~$B\coloneqq\prod_{j=1}^{r} b_j$ and~$N\coloneqq\min_j (N_j)$. Then \textbf{Condition~MD2}
for the set~$\rho^{(1)},\dots,\rho^{(r)}$ consists in:
\begin{enumerate}[1)]
\addtocounter{enumi}{-1}
\item $B\ne 0$, and the parameters~$\alpha_j$ for the Meixner weights satisfy Condition~D2;
\item $r_C = 0 \ \&\ r_M>0 \implies |B|<1$;
\item $r_K = r \implies B \ne (-1)^r$;
\item $r_H > 0 \implies r_K = r - 1%
    \ \ \&\ \ B = (-1)^{r-1}%
    \ \ \&\ \
    \alpha_r,\beta_r\notin\{0,-1,\dots,1-N_r\}$\\
    $\phantom{r_H > 0 \implies{}}\&\ \ \alpha_r + \beta_r + rn -\sum_{j=1}^{r-1}
    N_j\notin\{0,-1,\dots,-n\}$ \ for all ~$n\in\{0,1,\dots,N-1\}$,\\ where~$\rho^{(r)}$ stands
    for the (only) Hahn weight.\label{item:Hahn}
\end{enumerate}
In particular, the
set~$\rho^{(1)},\dots,\rho^{(r)}$ cannot contain more than one Hahn weight under Condition~MD2.
If the Hahn weight is present, then its Condition~D2 is replaced by~\ref{item:Hahn}), and the
remaining weights must be the Kravchuk weights.

Now, we are ready to state our main result.
\begin{theorem}\label{th:main}
    Let~$\rho^{(1)},\dots,\rho^{(r)}$ be a set of discrete weights satisfying Condition~MD2.
    Let~$\gamma_1,\dots,\gamma_r$ be real numbers such that none of the
    differences~$\gamma_j-\gamma_k$ and~$\gamma_j-\gamma_k-\alpha_j$
    and~$\gamma_j-\gamma_k-\beta_k$ is integer for~$j\ne k$. Put
    \[
        R(x)\coloneqq\prod_{j=1}^r\rho^{(j)}(x-\gamma_j),
        \quad
        R_n(x)\coloneqq\prod_{j=1}^r\rho_n^{(j)}(x-\gamma_j).
    \]
    Then for~$n\le N$ the Rodrigues formula
    \begin{equation}\label{eq:md_Rodrigues}
        P_n(x) R (x) = \nabla^n R_n(x),
    \end{equation}
    determines a polynomial~$P_n(x)$ of degree~$rn$ that satisfies the following orthogonality
    conditions:
    \begin{equation}\label{eq:orthog_md}
        \sum_{x\in\gamma_j+S_{\rho^{(j)}}} P_n(x)  x^k R(x)=0,
        \quad
        k=0,\dots,n-1,
        \quad
        j=1,\dots,r.
    \end{equation}
\end{theorem}
That is to say, under the conditions of Theorem~\ref{th:main}, the polynomial~$P_n$ determined
by the Rodrigues formula~\eqref{eq:md_Rodrigues} is a multiple orthogonal polynomial:
\begin{equation*}
    \int P_n(x)  x^k \, d\mu_j(x)=0,
    \quad
    k=0,\dots,n-1,
    \quad
    j=1,\dots,r
\end{equation*}
with respect to~$r$ discrete (generally speaking, complex-valued) measures
\begin{equation*}
    \mu_j(y)=\sum_{x\in\gamma_j+S_{\rho^{(j)}}} R(x)\delta(y-x)
\end{equation*}
with the pairwise disjoint supports~$\supp \mu_j(y)=\gamma_j+S_{\rho^{(j)}}$ and the common
analytic weight function~$R$.

Observe that the weight function~$R$ satisfies the difference Pearson equation
\begin{equation}\label{eq:md_Pearson}
    \Delta \big( \sigma(x) R(x) \big)
    = \tau(x) R(x)
    ,
\end{equation}
where~$\sigma$ and~$\tau$ are polynomials of degrees~$\deg\sigma\le r+1$ and~$\deg\tau=r$.

The next section is devoted to the proofs of Proposition~\ref{th:2} and Theorem~\ref{th:main}.
In Section~\ref{sect:classification}, we discuss the Conditions~MD1 for the case~$r=2$ which are
similar to Conditions~D1. That is to say, we discuss the conditions under which the
measures~$\mu_j$ are positive. On this way, we introduce several new classes of multiple
orthogonal polynomials with respect to the products of classical weights.

\section{Proof of main results}

\subsection{Proof of Proposition~\ref{th:2}}

Proposition~\ref{th:2} is a particular case of Theorem~\ref{th:main}. To give here a brief proof
of Proposition~\ref{th:2}, we need a simple lemma which can be verified by a straightforward
calculation: 
\begin{lemma}\label{lemma:diff_poly}
    Let~$q_1, q_2, p$ be monic polynomials:
    \[
        q_j(x)=a x^{r+1}+b_jx^{r}+\dots,
        \quad j=1,2;
        \quad
        p(x)=x^{m}+\dots,
    \]
    then
    \[
        q_1(x)p(x)-q_2(x)p(x-1) = (b_1-b_2 + a m) x^{r+m}+\dots,
    \]
    where the ellipses stand for the terms of lower degrees.
\end{lemma}

We also need the formula for summation by parts:
\[
    \sum_{x=0}^N g(x-1)\nabla f(x) = g(N) f(N) - g(-1)f(-1) - \sum_{x=0}^N f(x)\nabla g(x),
\]
as well as the following notation (see table~\ref{tab:Dclassical}):
\begin{equation}\label{eq:uv_def}
    u_m(x)\coloneqq \frac{\rho_{m+1}(x)}{\rho_m(x)}
    \quad\text{and}\quad
    v_m(x)\coloneqq \frac{\rho_{m+1}(x-1)}{\rho_m(x)}.
\end{equation}

\begin{lemma}
    Let~$\rho(x)$ be one of the classical weights from Table~\ref{tab:Dclassical} such that
    Condition~D2 is satisfied. Then, for all~$m=0,1,\dots,n$ and~$n=0,1,\dots,N$ the equality
    \begin{equation}\label{eq:our_discr_Rodrigues}
        \rho_{n-m}(x)P^{(n)}_m(x) = \nabla^m\rho_n(x),
    \end{equation}
    determines a polynomial~$P^{(n)}_m(x)$ of degree~$m$ orthogonal with the weight~$\rho_{n-m}$
    on~$S_{\rho}=\{0,1,\dots,N\}$. (In the Charlier and Meixner cases, we
    put~$N=+\infty$.)
\end{lemma}

\begin{proof}
    Induction on~$m$. If~$m=0$, the formula~\eqref{eq:our_discr_Rodrigues} yields the
    polynomial~$P^{(n)}_0(x)\equiv 1$ of zero degree, for which the orthogonality conditions are
    trivial.

    Now, for~$m\ge 1$ let~$P^{(n)}_{m-1}$ determined by~\eqref{eq:our_discr_Rodrigues} be a
    polynomial of degree~$m-1$ such that
    \begin{equation}\label{eq:orthog_md_dp}
        \sum_{x\in S_{\rho}} P^{(n)}_{m-1}(x)  x^k \rho_{n-m}(x)=0,
        \quad
        k=0,\dots,m-2.
    \end{equation}
    Then~$P^{(n)}_{m}(x)$ and~$P^{(n)}_{m-1}(x)$ turn to be related through
    \[
        \rho_{n-m}(x)P^{(n)}_m(x)
        = \nabla^m\rho_n(x)
        = \nabla\left[ P^{(n)}_{m-1}(x) \rho_{n-m+1}(x)\right].
    \]
    First, we need to verify that~$P^{(n)}_m(x)$ is a polynomial of degree~$m$. The last
    expression implies that
    \begin{equation}\label{eq:poly_relation}
        P^{(n)}_m(x) = u_{n-m}(x) P^{(n)}_{m-1}(x) -  v_{n-m}(x) P^{(n)}_{m-1}(x-1).
    \end{equation}
    As is seen from Table~\ref{tab:Dclassical}, the functions~$u_{n-m}(x)$ and~$v_{n-m}(x)$ are
    linear for the non-Hahn cases. Under Condition~D2, cancellation of the leading term does not
    occur, and hence the left-hand side of~\eqref{eq:poly_relation} is a polynomial of
    degree~$m$.

    In the Hahn case, the polynomials~$u_{n-m}(x)$ and~$v_{n-m}(x)$ are quadratic, and their
    leading coefficients coincide. That is, their difference is linear:
    \[
        v_m(x)-u_m(x)
        = \frac{(\beta + \alpha +2m)}{(\alpha+m)(\beta+m)}\cdot x + \frac{m - N}{\beta+m},
    \]
    since the contributions of their leading coefficients cancel each other independently of the
    parameters:
    \[
        [x^{m+1}]P^{(n)}_m(x)
        = [x^{m+1}]\left( u_{n-m}(x) P_{m-1}^{(n)}(x) - v_{n-m}(x)P_{m-1}^{(n)}(x-1)\right)
        =0.
    \]
    The next coefficient may be calculated with the help of Lemma~\ref{lemma:diff_poly}:
    \[
        \begin{aligned}
        {[x^{m}]}P^{(n)}_m(x)
        &= [x^{m}]\left( u_{n-m}(x) P_{m-1}^{(n)}(x) - v_{n-m}(x)P_{m-1}^{(n)}(x-1)\right)
        \\
        &= -\frac{\beta +\alpha + 2(n-m)+(m-1)}{(\alpha+n-m)(\beta+n-m)}\cdot [x^{m-1}]P^{(n)}_{m-1}(x).
        \end{aligned}
    \]
    This coefficient does not cancel when Condition~D2 from Table~\ref{tab:Dclassical} holds.
    Consequently, the right-hand side of~\eqref{eq:poly_relation} is a polynomial of degree~$m$
    in the Hahn case as well.

    The next aim is to check that~$P^{(n)}_m(x)$ satisfies the orthogonality conditions.
    Let~$S_\rho=\{0,1,\dots,N\}$, where~$N\in\mathbb{N}\cup \{\infty\}$, then for~$k\le m-1$
    \begin{equation}\label{eq:check_orthog}
        \begin{aligned}
            \sum_{x=0}^N P^{(n)}_m(x) x^k \rho_{n-m}(x)
            &=
            \sum_{x=0}^N x^k \nabla\big[P^{(n)}_{m-1}(x) \rho_{n-m+1}(x)\big]
            \\
            &=
            (x+1)^kP_{m-1}^{(n)}(x)\rho_{n-m+1}(x)\Big|_{x=-1}^{N}
            -
            \sum_{x=0}^N P_{m-1}^{(n)}(x)\rho_{n-m+1}(x) \nabla (x+1)^k
            .
        \end{aligned}
    \end{equation}
    By the induction hypothesis, the sum on the right-hand side is zero.
    Moreover,
    \[
        \rho_{n-m+1}(-1)=\rho_{n-m+1}(N)=0
        \quad\text{when}\quad
        m\le n,
    \]
    since the reciprocal Gamma function~$1/\Gamma(x)$ has zeros at integer non-positive points.%
    \footnote{More generally, the weight~$\rho_{n-m+1}(x)$ vanishes for each
        integer~$x\in[-n+m-1,-1]\cup[N-n+m,N]$, which explains why one may use~$S_\rho$ instead
        of~$\{0,1,\dots,N-n+m-1\}$ as the of orthogonality region for~$\rho_{n-m+1}$.}
    For the Charlier and Meixner weights corresponding to~$N=\infty$, Condition~D2 yields a
    superpolynomial rate of vanishing of~$\rho_{n-m+1}(x)$ as~$x\to\infty$. Therefore, the first
    term on the right-hand side is also zero. The inductive step is complete, so the lemma is
    proved.
\end{proof}

Proposition~\ref{th:2} follows from this lemma for~$m=n$, because~$\rho=\rho_0$.

\subsection{Proof of Theorem~\ref{th:main}}
Theorem~\ref{th:main} is a particular case of the following lemma:
\begin{lemma}
    Under the conditions of Theorem~\ref{th:main}, the equality
    \begin{equation}\label{eq:our_mdiscr_Rodrigues}
        R_{n-m}(x) P^{(n)}_m(x) = \nabla^{m}R_n(x)
    \end{equation}
    for~$n\le N$ and~$m\in\{0,\dots,n\}$ determines a polynomial~$P^{(n)}_m(x)$ of degree~$rm$
    which satisfies the conditions of multiple orthogonality with the weight~$R_{n-m}$ on the
    sets~$\gamma_j+S_{\rho^{(j)}}$:
    \[
        \sum_{x\in\gamma_j+S_{\rho^{(j)}}} P^{(n)}_m(x)  x^k R_{n-m}(x)=0,
        \quad
        k=0,\dots,m-1,
        \quad
        j=1,\dots,r.
    \]
\end{lemma}
\begin{proof}
    We are using the induction on~$m$. Similarly to the case of orthogonality with respect to
    one measure, the equality~\eqref{eq:our_mdiscr_Rodrigues} with~$m=0$ is satisfied only by
    the polynomial~$P^{(n)}_m\equiv 1$ of degree~$0$. Observe that
    \[
        R_{n-m}(x) P^{(n)}_m(x) = \nabla^{m}R_n(x)= \nabla\left[P^{(n)}_{m-1}(x) R_{n-m+1}(x)\right]
    \]

    As an induction hypothesis, assume that~$P^{(n)}_{m-1}\not\equiv 0$ for~$m\in\{1,2,\dots,n\}$ is
    a polynomial of degree~$r(m-1)$ such that
    \[
        \sum_{x\in\gamma_j+S_{\rho^{(j)}}} P^{(n)}_{m-1}(x)  x^k R_{n-m+1}(x)=0,
        \quad
        k=0,\dots,m-2,
        \quad
        j=1,\dots,r.
    \]
    Let us check that~$P^{(n)}_{m}(x)$ is a polynomial of degree~$rm$. Indeed,
    \[
        \begin{aligned}
            P^{(n)}_m(x) %
            &= \frac{ R_{n-m+1}(x)}{R_{n-m}(x)} P^{(n)}_{m-1}(x)
            - \frac{ R_{n-m+1}(x-1)}{R_{n-m}(x)} P^{(n)}_{m-1}(x-1)\\
            &= \prod_{j=1}^r\frac{\rho^{(j)}_{n-m+1}(x-\gamma_j)}{\rho^{(j)}_{n-m}(x-\gamma_j)} P^{(n)}_{m-1}(x)
            - \prod_{j=1}^r\frac{\rho^{(j)}_{n-m+1}(x-\gamma_j-1)}{\rho^{(j)}_{n-m}(x-\gamma_j)} P^{(n)}_{m-1}(x-1)\\
            &= \prod_{j=1}^ru^{(j)}_{n-m}(x-\gamma_j) \cdot P^{(n)}_{m-1}(x)
            - \prod_{j=1}^rv^{(j)}_{n-m}(x-\gamma_j) \cdot P^{(n)}_{m-1}(x-1),
        \end{aligned}
    \]
    where~$u_m^{(j)}$ and~$v_m^{(j)}$ are defined analogously to~\eqref{eq:uv_def}. Therefore,
    the following representation is true:
    \begin{equation}\label{eq:md_pm_mpm1}
            P^{(n)}_m(x) 
            = U_{n-m}(x) \cdot P^{(n)}_{m-1}(x)
            - V_{n-m}(x) \cdot P^{(n)}_{m-1}(x-1),
    \end{equation}
    where
    \[
        U_m(x)\coloneqq \prod_{j=1}^ru^{(j)}_{m}(x-\gamma_j)
        \quad
        \text{and}
        \quad
        V_m(x)\coloneqq \prod_{j=1}^rv^{(j)}_{m}(x-\gamma_j).
    \]
    Since~$U_{n-m}(x)$ and~$V_{n-m}(x)$ are polynomials,~$P^{(n)}_m(x)$ is also a polynomial.

    To determine its degree, first consider the case~$r_H=0$, that is when there is no Hahn
    weights among~$\rho^{(1)},\dots,\rho^{(r)}$. If so, then all~$u^{(j)}_{n-m}$
    and~$v^{(j)}_{n-m}$ are linear, and hence~$\deg P^{(n)}_{m}\le rm$.

    When one of the weights~$\rho^{(1)},\dots,\rho^{(r)}$ is a Charlier weight,
    then~$\deg V_{n-m}=r>\deg U_{n-m}$. Consequently, if~$r_C>0$,
    then~$\deg P_m^{(n)}=r+\deg P_{m-1}^{(n)}=rm$.

    If~$r_C=0$ and~$r_M>0$, then
    \[
        [x^r] U_{n-m}=(-1)^{r_K},
        \quad
        [x^r] V_{n-m}=\frac 1 B.
    \]
    By Condition~MD2 in this case~$|B|<1$, so the leading coefficients of~$U_{n-m}$
    and~$V_{n-m}$ are different in absolute value, and thus~$\deg P_{m}^{(n)}=rm$.

    If~$r_H=r_C=r_M=0$, that is~$r=r_K$, then
    \[
        [x^r] U_{n-m}=(-1)^{r},
        \quad
        [x^r] V_{n-m}=\frac 1 B.
    \]
    Condition~MD2 here implies~$B\ne (-1)^{r}$, and hence~$\deg P_{m}^{(n)}=rm$ as well.

    Now, let~$r_H>0$. By Condition~MD2, we have~$r_H=1$, ~$r_K=r-1$ and~$B=(-1)^{r-1}$. Without
    loss of generality assume that the Hahn weight is denoted by~$\rho^{(r)}$. The
    polynomials~$U_{n-m}$ and~$V_{n-m}$ both have the degree~$r+1$. Their leading coefficients
    however coincide:
    \[
        (\alpha+n-m)(\beta+n-m)\cdot [x^{r+1}]U_{n-m}=(-1)^r,
        \quad
        (\alpha+n-m)(\beta+n-m)\cdot [x^{r+1}]V_{n-m}=-\frac 1{B}=(-1)^{r}.
    \]
    To calculate the coefficients of~$U_{n-m}$ and~$V_{n-m}$ at~$x^r$, note that these
    polynomials can be written as
    \[
        \begin{aligned}
            U_{n-m}&=\frac{x-\gamma_r+\alpha+(n-m)}{(\alpha+n-m)(\beta+n-m)}
            \prod_{j=1}^{r}\big(N_j-(n-m)-x+\gamma_j\big),
            \\
            V_{n-m}&=\frac{N_r-x+\gamma_r+\beta}{(\alpha+n-m)(\beta+n-m)B} \prod_{j=1}^{r}(x-\gamma_j).
        \end{aligned}
    \]
    Vieta's theorem then implies
    \[
        \begin{aligned}
            (\alpha+n-m)(\beta+n-m)\cdot{[x^{r}]}U_{n-m}
            &=(-1)^{r-1}\left(\sum_{j=1}^r N_j + \sum_{j=1}^r\gamma_j + \gamma_r - (r+1)(n-m)
                -\alpha \right)
            ,\\
            (\alpha+n-m)(\beta+n-m)\cdot[x^{r}]V_{n-m}
            &=\frac 1B \left(\sum_{j=1}^r\gamma_j + \gamma_r +N_r+\beta \right)
            =(-1)^{r-1}\left(\sum_{j=1}^r\gamma_j + \gamma_r +N_r+\beta \right).
        \end{aligned}
    \]
    From~\eqref{eq:md_pm_mpm1} and Lemma~\ref{lemma:diff_poly}, we see that~$P_n^{(m)}$ has
    degree~$rm$ and its leading coefficient satisfies the equality:
    \[
        \begin{aligned}
        (-1)^{r}\,[x^{rm}]\, P_m^{(n)} (x)
        &=\left(
            \alpha + \beta + (r+1)(n-m) -\sum_{j=1}^{r-1} N_j
            +r(m-1)
        \right)\frac{[x^{r(m-1)}]\, P_{m-1}^{(n)} (x)}{(\alpha+n-m)(\beta+n-m)}
        \\
        &=\left(
            \alpha + \beta + r(n-1) + (n - m) -\sum_{j=1}^{r-1} N_j
        \right)\frac{[x^{r(m-1)}]\, P_{m-1}^{(n)} (x)}{(\alpha+n-m)(\beta+n-m)}
        ;
        \end{aligned}
    \]
    it does not vanish when Condition~MD2 is satisfied.

    To verify the orthogonality conditions, it is enough to use the precisely the same approach
    as for~$r=1$. The conditions for the shifts~$\gamma_j$ in Theorem~\ref{th:main} guarantee
    that~$R_{n-m+1}(-1+\gamma_j)=R_{n-m+1}(N+\gamma_j)=0$ with~$m\in \{1,\dots,n\}$: the
    denominators of~$R_{n-m+1}$ have at these points poles, while their numerators are regular.
    As a result, the term with the values on the boundary is equal to zero on the inductive
    step, see~\eqref{eq:check_orthog}.
\end{proof}

For particular cases of discrete multiple orthogonal polynomials, the conditions of
Theorem~\ref{th:main} on the shifts~$\gamma_j$ and the parameters~$\alpha_j, \beta_j$ may be
somewhat relaxed. The next section is devoted to particular examples corresponding to~$r=2$
weights.

\section{Classification of polynomials of multiple discrete orthogonality}
\label{sect:classification}

Further we restrict ourselves to the case~$r=2$. Our aim is to study the polynomials~$P_n$ of
degree~$2n$ orthogonal with respect to the pair of discrete measures~$\mu_1$ and~$\mu_2$:
\begin{equation}\label{eq:orth_mdiscr2}
    \int P_n(x)  x^k \, d\mu_j(x)=0,
    \quad
    k=0,\dots,n-1,
    \quad
    j=1,2.
\end{equation}
In particular, we discuss the conclusions from Theorem~\ref{th:main} when the measures~$\mu_1$
and~$\mu_2$ are positive. Depending on whether the convex hulls
of~$\supp\mu_1=\gamma_1+S_{\rho^{(1)}}$ and~$\supp\mu_2=\gamma_2+S_{\rho^{(2)}}$ are disjoint or
not, we point out two cases:
\begin{itemize}
\item Case~I: $\operatorname{conv}
    (\gamma_1+S_{\rho^{(1)}}) \cap \operatorname{conv}(\gamma_2+S_{\rho^{(2)}}) \ne \varnothing$;
\item Case~II: $\operatorname{conv}
    (\gamma_1+S_{\rho^{(1)}}) \cap \operatorname{conv}(\gamma_2+S_{\rho^{(2)}}) = \varnothing$.
\end{itemize}

\subsection{Case I}

\subsubsection*{Charlier-Charlier polynomials}
Let~$B\coloneqq b>0$, and let~$\gamma_1,\gamma_2$ be real numbers such
that~$0<|\gamma_1-\gamma_2|<1$. Then the function
\[
    R(x)\coloneqq\frac{b^x}{\Gamma(x-\gamma_1+1)\Gamma(x-\gamma_2+1)}
\]
is a product of two Charlier weights on the lattices shifted by~$\gamma_1$ and~$\gamma_2$. It is
positive for~$x-\gamma_j\in\mathbb Z_+$, ~$j=1,2$. Determine a couple of discrete measures:
\begin{equation}\label{eq:chch_measures}
    \mu_j(y)\coloneqq \sum_{x\in\gamma_j+\mathbb Z_+} R(x)\delta(y-x),
    \quad
    j=1,2.
\end{equation}

\begin{figure}[ht]
    \caption{Zeros of the Charlier-Charlier polynomials when $b=1$, $\gamma_1=0$, $\gamma_2=\frac 12$.}%
    \label{fig:3}
    \includegraphics[scale=.8]{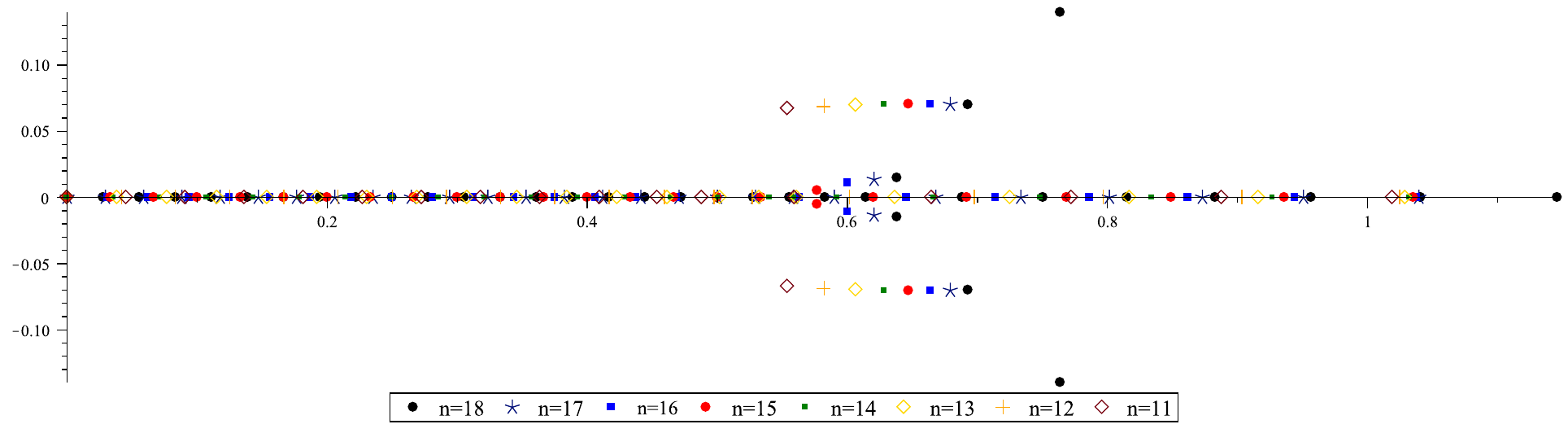}
\end{figure}

Then the \emph{Charlier-Charlier polynomial}~$P_n$ of degree~$2n$ is determined through the formula:
\[
    P_n(x) R(x)=\nabla^n R(x).
\]
This polynomial satisfies~$2n$ orthogonality conditions~\eqref{eq:orth_mdiscr2} with respect to
the measures~\eqref{eq:chch_measures}.

\subsubsection*{Charlier-Meixner polynomials}
Suppose that~$b>0$, ~$\alpha>0$ and~$\gamma_1,\gamma_2\in\mathbb R$ are such
that~$0<|\gamma_1-\gamma_2|<1$ and~$\gamma_1-\gamma_2<\alpha$. The following function is
positive for~$x\ge\min (\gamma_1,\gamma_2)$:
\[
    R^{(\alpha)}(x)=\frac{b^x \Gamma(x-\gamma_2+\alpha)}
    {\Gamma(x-\gamma_1+1)\Gamma(x-\gamma_2+1)}.
\]
On integer lattices with shifts, consider two discrete measures:
\begin{equation}\label{eq:chmx_measures}
    \mu_j(y)\coloneqq \sum_{x\in\gamma_j+\mathbb Z_+} R^{(\alpha)}(x)\delta(y-x),
    \quad
    j=1,2.
\end{equation}
Then the polynomial~$P_n$ determined by the Rodrigues formula
\[
    P_n(x) R^{(\alpha)}(x)=\nabla^n R^{(\alpha+n)}(x)
\]
has degree~$2n$ and satisfies the orthogonality relations~\eqref{eq:orth_mdiscr2} with respect
to the measures~\eqref{eq:chmx_measures}. Polynomials of this form we call the
\emph{Charlier-Meixner polynomials}.

\subsubsection*{Meixner-Sorokin polynomials}
Let the numbers~$b\in(0,1)$, ~$\alpha_1,\alpha_2>0$ and~$\gamma_1,\gamma_2\in\mathbb R$ satisfy
the inequalities~$0<|\gamma_1-\gamma_2|<1$ and~$-\alpha_2<\gamma_1-\gamma_2<\alpha_1$. Consider
the following function positive for~$x\ge\min (\gamma_1,\gamma_2)$:
\[
    R^{(\alpha_1,\alpha_2)}(x)\coloneqq b^x\frac{\Gamma(x-\gamma_1+\alpha_1)\Gamma(x-\gamma_2+\alpha_2)}
    {\Gamma(x-\gamma_1+1)\Gamma(x-\gamma_2+1)}.
\]
The function~$R^{(\alpha_1,\alpha_2)}$ is a product of two Meixner weights.

On each of the lattices~$\gamma_j+\mathbb Z_+$ we introduce the measure~$\mu_j$:
\begin{equation}\label{eq:mxsr_measures}
    \mu_j(y)\coloneqq \sum_{x\in\gamma_j+\mathbb Z_+} R^{(\alpha_1,\alpha_2)}(x)\delta(y-x),
    \quad
    j=1,2.
\end{equation}
Then the polynomial~$P_n$ of degree~$2n$ given by:
\[
    P_n(x) R^{(\alpha_1,\alpha_2)}(x)=\nabla^n R^{(\alpha_1+n,\alpha_2+n)}(x)
\]
satisfies the orthogonality conditions~\eqref{eq:orth_mdiscr2}, where the measures are
introduced in~\eqref{eq:mxsr_measures}. This kind of polynomials first appeared in the
paper~\cite{Sor2010}, so we call them the \emph{Meixner-Sorokin polynomials}.

\subsubsection*{Kravchuk-Kravchuk polynomials}
Let the parameters~$b>0$, ~$N_1,N_2\in\mathbb N$ and~$\gamma_1,\gamma_2\in\mathbb R$ satisfy one
of two inequalities:
\begin{itemize}
\item $N_1=N_2$, ~$0<|\gamma_1-\gamma_2|<1$;
\item $N_1=N_2-1$, ~$0<\gamma_1-\gamma_2<1$.
\end{itemize}
Note that, as above, these inequalities mean that the nodes of these lattices are interlacing:
between each two nodes of the first lattice there is precisely one node of the other lattice and
vice versa. In this case, the lattices are~$\gamma_j+\{0,1,\dots,N_j\}$. Consider the product of
two Kravchuk's weights:
\[
    R_{N_1,N_2}(x)\coloneqq\frac{b^x}
    {\Gamma(x-\gamma_1+1)\Gamma(N_1-x+\gamma_1+1)}
    \cdot \frac{1}
    {\Gamma(x-\gamma_2+1)\Gamma(N_2-x+\gamma_2+1)},
\]
and the pair of measures:
\begin{equation}\label{eq:krkr_measures}
    \mu_j(y)\coloneqq \sum_{x-\gamma_j=0}^{N_j} R_{N_1,N_2}(x)\delta(y-x),
    \quad
    j=1,2.
\end{equation}

For~$n\in\mathbb Z_+$, ~$n\le N_1$, determine the polynomial~$P_n$ of degree~$2n$ by:
\[
    P_n(x) R_{N_1,N_2}(x)=\nabla^n R_{N_1-n,N_2-n}(x).
\]

Then~$P_n$ satisfies the orthogonality conditions~\eqref{eq:orth_mdiscr2} with respect to the
measures~$\mu_1,\mu_2$ from~\eqref{eq:krkr_measures}. We will call~$P_n$ the
\emph{Kravchuk-Kravchuk polynomial}.

\begin{figure}[ht]
    \caption{Zeros of the Kravchuk-Kravchuk polynomials of index~$n>10$ (on the left)
        and~$n\le10$ (on the right) for the parameters
        \[b=2,\quad N_1=N_2=20,\quad \gamma_1=0,\quad \gamma_2=\frac 12.\]}%
    \label{fig:4}
    \includegraphics[height=8cm]{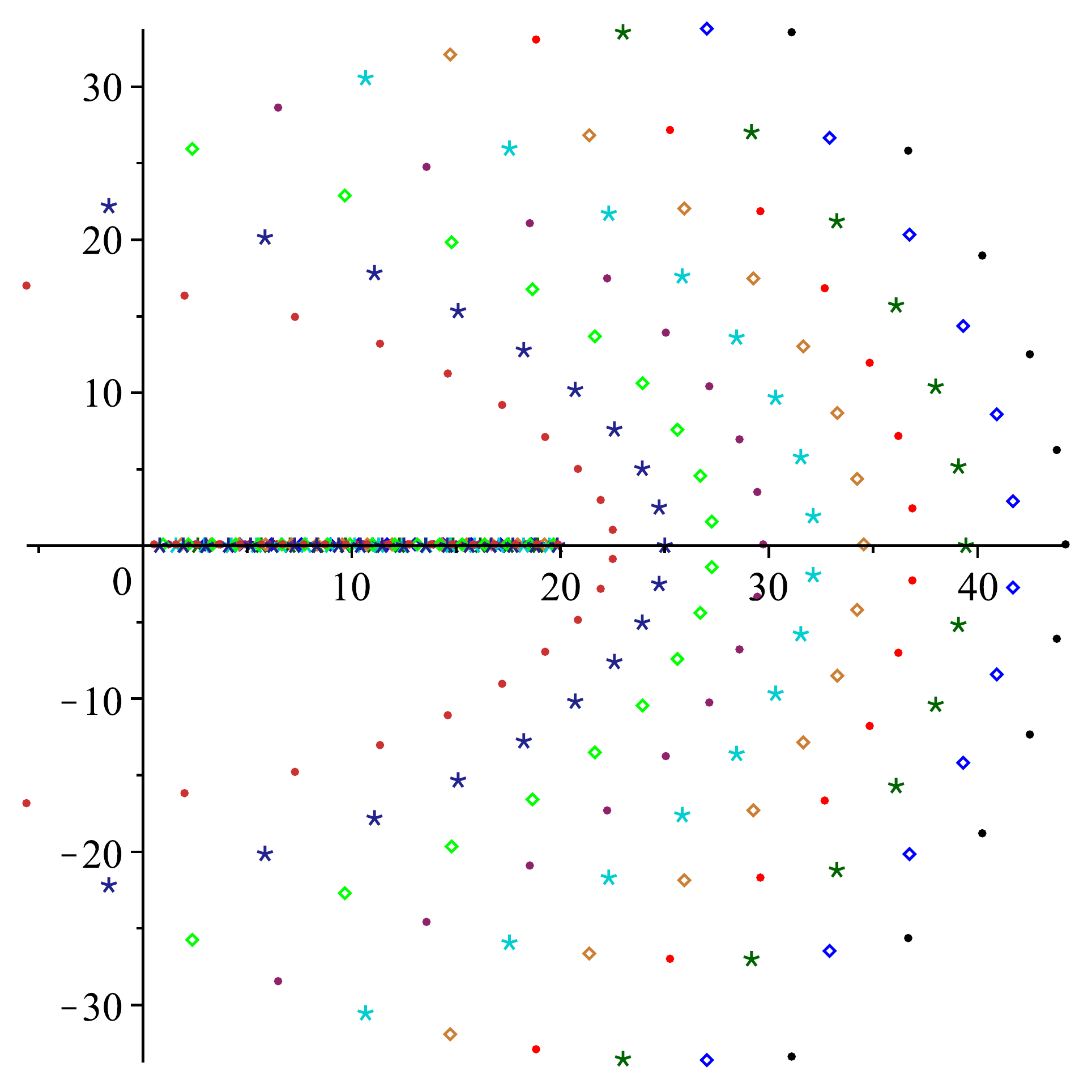}
    \includegraphics[height=8cm]{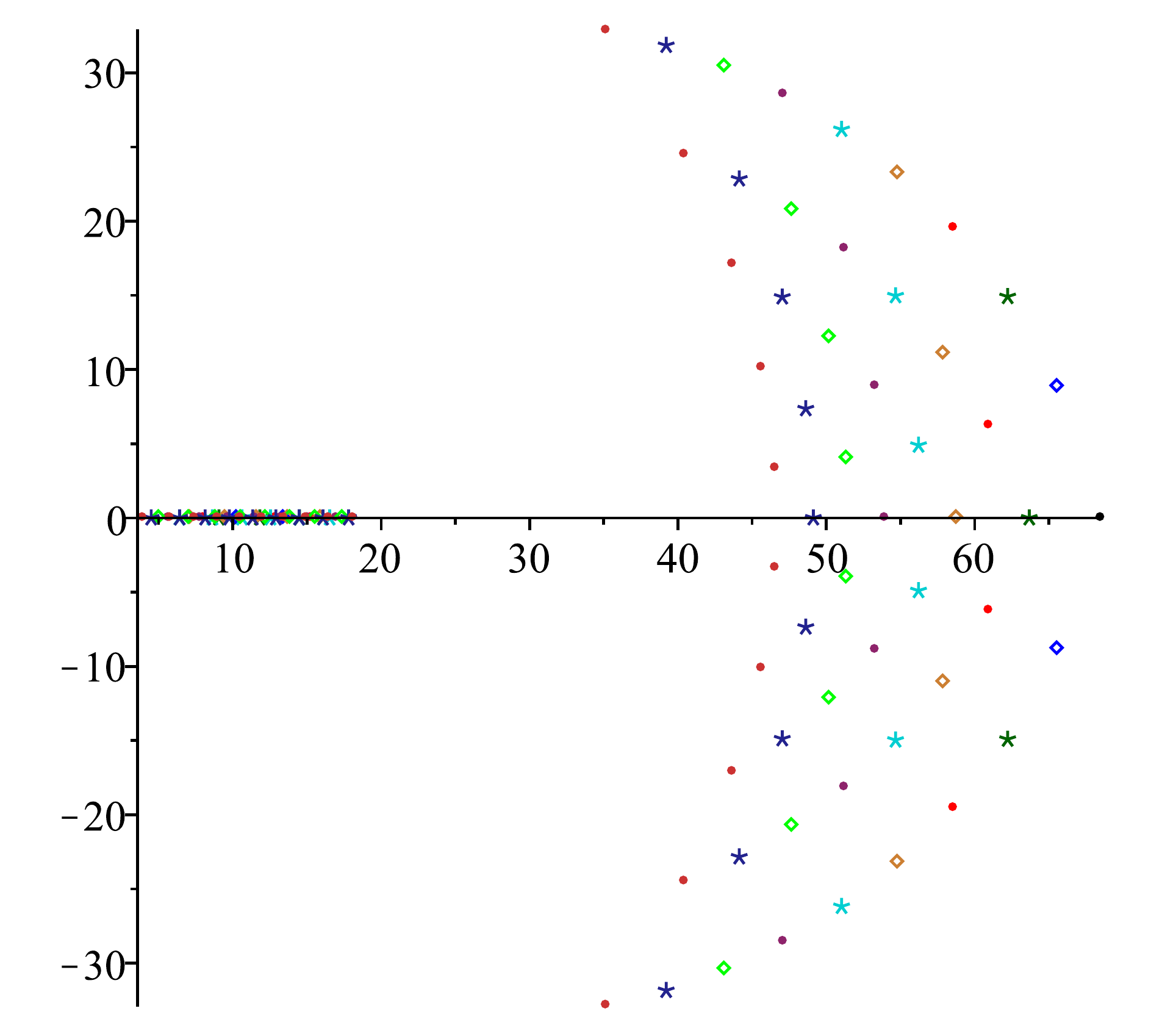}
\end{figure}

\subsubsection*{Kravchuk-Hahn I polynomials}
Fix~$B\coloneqq -1$
and let the parameters~$N_1,N_2\in\mathbb N$
and~$\alpha$, $\beta$, $\gamma_1$, $\gamma_2\in\mathbb R$ satisfy one of the following six conditions:
\begin{enumerate}[(a)]\itemsep7pt
\item\label{item:hk_a}
    $N_1=N_2$, ~$0<|\gamma_1-\gamma_2|<1$, ~$\alpha > \max\{0,\gamma_2-\gamma_1\}<-\beta-N_2$,
    and~$\beta-\gamma_1+\gamma_2,\beta \notin-\mathbb{Z}_+$;
\item $N_1=N_2$, ~$0<|\gamma_1-\gamma_2|<1$, ~$\beta > \max\{0,\gamma_1-\gamma_2\}<-\alpha-N_2$,
    and~$\alpha+\gamma_1-\gamma_2,\alpha \notin-\mathbb{Z}_+$;
\item $N_1=N_2-1$, $0<\gamma_1-\gamma_2<1$, ~$\alpha > 0<-\beta-N_2$,
    and~$\beta-\gamma_1+\gamma_2,\beta \notin-\mathbb{Z}_+$;
\item $N_1=N_2-1$, $0<\gamma_1-\gamma_2<1$, ~$\beta > 0<-\alpha-N_2$,
    and~$\alpha+\gamma_1-\gamma_2,\alpha \notin-\mathbb{Z}_+$;
\item $N_1=N_2+1$, $0<\gamma_2-\gamma_1<1$, ~$\alpha > \gamma_2-\gamma_1 <-\beta-N_2$,
    and~$\beta-\gamma_1+\gamma_2,\beta \notin-\mathbb{Z}_+$;
\item\label{item:hk_f}
    $N_1=N_2+1$, $0<\gamma_2-\gamma_1<1$, ~$\beta > 1-(\gamma_2-\gamma_1) <-\alpha-N_2$,
    and~$\alpha+\gamma_1-\gamma_2,\alpha \notin-\mathbb{Z}_+$.
\end{enumerate}

Consider the function
\[
    G_{N_1,N_2}^{(\alpha,\beta)}(x)
    \coloneqq
    \frac{1}
    {\Gamma(x-\gamma_1+1)\Gamma(N_1-x+\gamma_1+1)}
    \cdot \frac{(\alpha)_{x-\gamma_2}(\beta)_{N_2-x+\gamma_2}}
    {\Gamma(x-\gamma_2+1)\Gamma(N_2-x+\gamma_2+1)}.
\]
Each of the conditions~(\ref{item:hk_a}--\ref{item:hk_f}) implies that,
for~$x\in [\gamma_j,\gamma_j+N_j]$, the denominator of~$G_{N_1,N_2}^{(\alpha,\beta)}$ is
positive, while the sign of the numerator alternates so that the
expression~$(-1)^{x-\gamma_j} G_{N_1,N_2}^{(\alpha,\beta)}(x)$ keeps its sign on the
lattice~$\gamma_j + \{0,1,\dots,N_j\}$, where~$j=1,2$. Hence,
\begin{equation}\label{eq:krHn_measures}
    \mu_j(y)\coloneqq \sum_{x-\gamma_j=0}^{N_j} (-1)^{x-\gamma_j} G_{N_1,N_2}^{(\alpha,\beta)}(x)\delta(y-x),
    \quad
    j=1,2,
\end{equation}
are two discrete measures of constant sign. Given~$n\in\mathbb{Z}_+$ such
that~$n\le N\coloneqq\min\{N_1,N_2\}$, let us call the polynomial~$P_n$ satisfying the relation
\begin{equation}\label{eq:krHn_Rodrigues}
    P_n(x) (-1)^x G_{N_1,N_2}^{(\alpha,\beta)}(x)=\nabla^n \left[ (-1)^x
        G_{N_1-n,N_2-n}^{(\alpha+n,\beta+n)}(x)\right]
\end{equation}
the \emph{Kravchuk-Hahn I polynomial}.
\begin{figure}[ht]
    \caption{Zeros of the Kravchuk-Hahn I polynomials of index~$n\le 10$ (on the
        top),~$10\le n\le 16$ (in the middle) and~$n> 17$ (in the bottom) for the parameters
        \[\alpha=1,\quad \beta=-\frac{92}3,\quad N_1=N_2=20,\quad \gamma_1=-\frac 13,
            \quad\gamma_2=0.\]}
    \label{fig:KHI}
    \includegraphics[width=15cm]{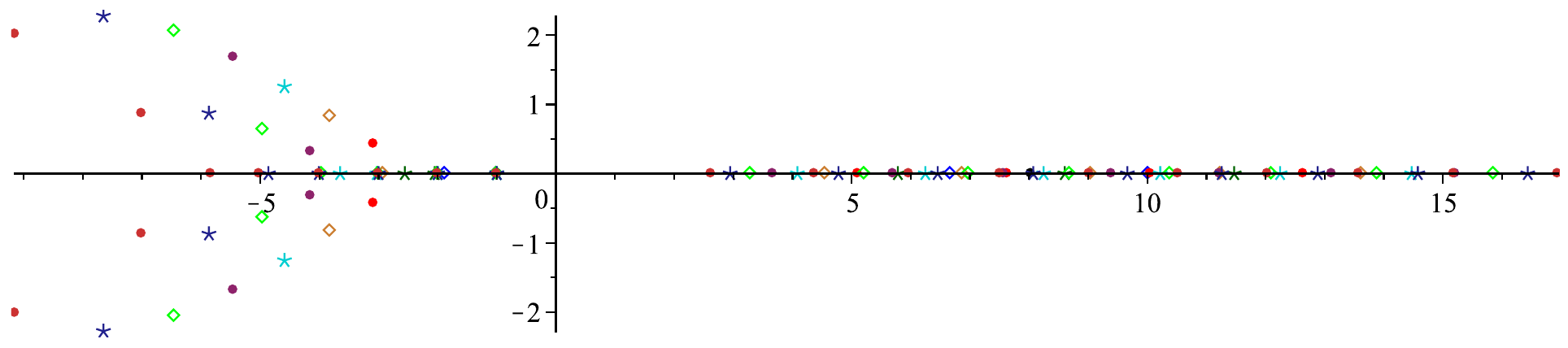}\\
    \includegraphics[width=15cm]{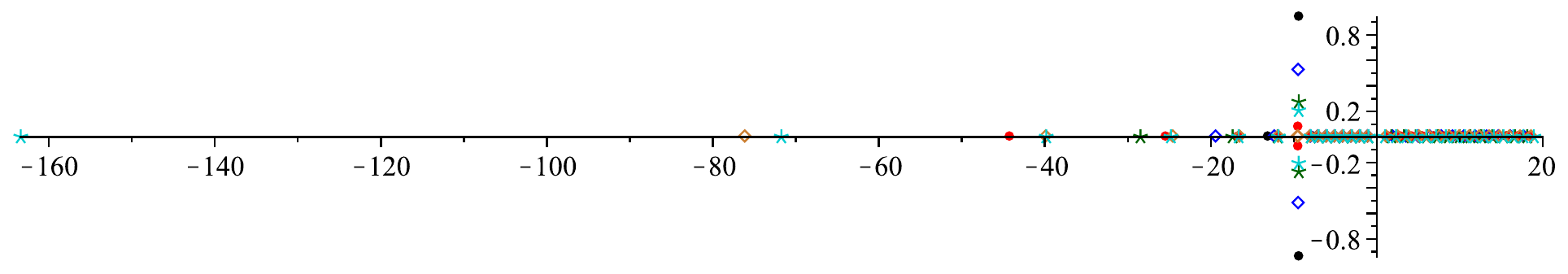}\\
    \includegraphics[width=15cm]{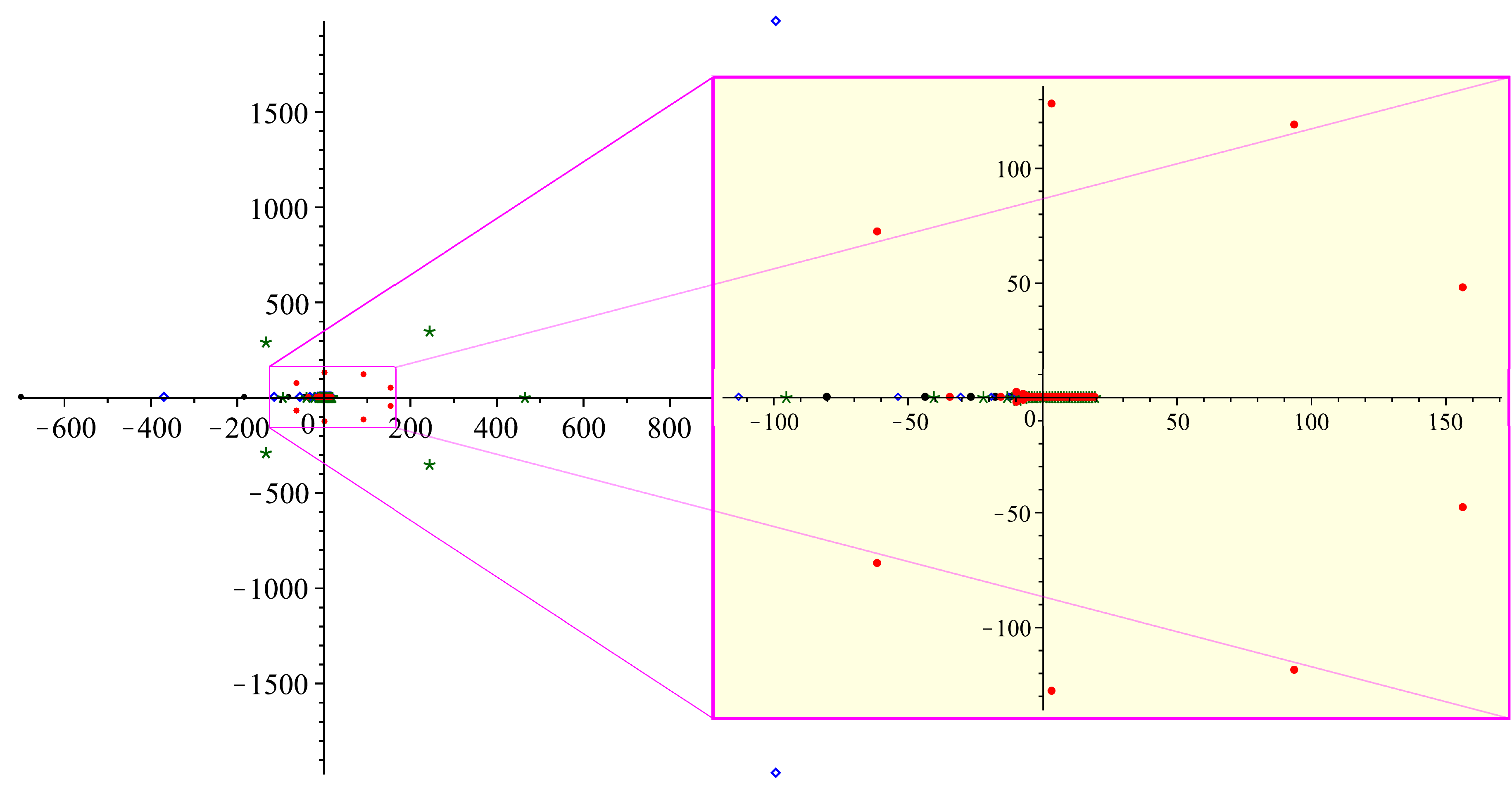}
\end{figure}
For each~$0\le n\le N$, it has degree~$2n$ and satisfies
the orthogonality conditions~\eqref{eq:orth_mdiscr2} with respect to the measures~$\mu_1,\mu_2$
from~\eqref{eq:krHn_measures} --- provided that Condition~MD2 is satisfied. Here this condition
is nontrivial:
\begin{equation}\label{eq:krHn_MD2}
    \alpha+\beta \notin \{N_1-3N+3,\ N_1-3N+4,\dots,\ N_1-3,\ N_1-2,\ N_1\}
    .
\end{equation}
See Figure~\ref{fig:KHI} for the example how zeros of the Kravchuk-Hahn I polynomials are
located on the complex plane.

We considered the case~I, when the convex hulls of supports of the measures~$\mu_1,\mu_2$ have
common points. More specifically, the points of~$\supp\mu_1$ and~$\supp \mu_2$ were interlacing.
It is easy to check that the interlacing property in case~I is necessary for the measures to be
of constant signs.

\subsection{Case II}
Now let us switch to the case when the convex hulls of supports of the measures have no common
points.
\subsubsection*{Charlier-Kravchuk polynomials}
Suppose that~$B\coloneqq -b$, ~$b>0$, ~$N\in\mathbb{N}$, and~$\gamma_1,\gamma_2\in\mathbb R$ are
such that~$\gamma_1-\gamma_2>N$ and~$\gamma_1-\gamma_2\notin\mathbb N$. Denote
\[
    G_N(x)\coloneqq \frac{1}
    {\Gamma(x-\gamma_1+1)\Gamma(x-\gamma_2+1)\Gamma(N-x+\gamma_2+1)}
\]
and consider two discrete measures with constant signs:
\begin{equation}\label{eq:chkr_measures}
    \begin{aligned}
        \mu_1(y)&\coloneqq \sum_{x-\gamma_1=0}^\infty (-b)^{x-\gamma_1} G_N(x)\delta(y-x),
        \\
        \mu_2(y)&\coloneqq \sum_{x-\gamma_2=0}^N (-b)^{x-\gamma_2} G_N(x)\delta(y-x).
    \end{aligned}
\end{equation}
Observe that the inequality~$\gamma_1-\gamma_2>N$ guarantees that convex hulls of the supports
of the measures~$\mu_1$ and~$\mu_2$ are disjoint.

For~$n\in\mathbb Z_+$ such that~$n\le N$, determine the polynomial~$P_n$ of degree~$2n$
by the formula
\[
    P_n(x) (-b)^x G_N(x)=\nabla^n \left[ (-b)^x G_N(x)\right].
\]
Then the orthogonality relations~\eqref{eq:orth_mdiscr2} are satisfied by~$P_n$ with the
measures~\eqref{eq:chkr_measures}. We call~$P_n$ the \emph{Charlier-Kravchuk polynomial}.

\subsubsection*{Meixner-Kravchuk polynomials}
This family is constructed analogously to the Charlier-Kravchuk case. Suppose that~$b\in(0,1)$,
~$N\in\mathbb{N}$, and~$\alpha,\gamma_1,\gamma_2\in\mathbb R$ are such that~$\alpha>0$
and~$\gamma_1-\gamma_2>N$, ~$\gamma_1-\gamma_2\notin\mathbb{N}$. Put
\[
    G_N^{(\alpha )}(x)\coloneqq \frac{\Gamma(x-\gamma_1+\alpha)}
    {\Gamma(x-\gamma_1+1) \Gamma(x-\gamma_2+1) \Gamma(N-x+\gamma_2+1)}
\]
and consider two discrete measures with constant signs:
\begin{equation}\label{eq:mxkr_measures}
    \begin{aligned}
        \mu_1(y)&\coloneqq \sum_{x-\gamma_1=0}^\infty (-b)^{x-\gamma_1} G_N^{(\alpha )}(x)\delta(y-x),
        \\
        \mu_2(y)&\coloneqq \sum_{x-\gamma_2=0}^N (-b)^{x-\gamma_2} G_N^{(\alpha )}(x)\delta(y-x).
    \end{aligned}
\end{equation}
The convex hulls of~$\supp\mu_1$ and~$\supp\mu_2$ have no common points.

As above for~$n\in\mathbb Z_+$, ~$n\le N$, determine the polynomial~$P_n$ of degree~$2n$:
\[
    P_n(x) (-b)^x G_N^{(\alpha )}(x)=\nabla^n \left[ (-b)^x G_{N-n}^{(\alpha +n)}(x)\right].
\]
The \emph{Meixner-Kravchuk polynomial}~$P_n$ is orthogonal~\eqref{eq:orth_mdiscr2} with respect
to the measures~\eqref{eq:mxkr_measures}.

\subsubsection*{Angelesco-Kravchuk polynomials}
Let the parameters~$b>0$, ~$N_1,N_2\in\mathbb{N}$, and~$\gamma_1,\gamma_2\in\mathbb R$ be such
that~$\gamma_1-\gamma_2>N_2$ and~$\gamma_1-\gamma_2\notin\mathbb{N}$. Denote
\[
    G_{N_1,N_2}(x)
    \coloneqq \frac{1}{\Gamma(x-\gamma_1+1) \Gamma(N-x+\gamma_1+1)}
    \cdot \frac{1}{\Gamma(x-\gamma_2+1) \Gamma(N-x+\gamma_2+1)}.
\]
Then the discrete measures 
\begin{equation}\label{eq:angkr_measures}
    \mu_j(y)\coloneqq \sum_{x-\gamma_j=0}^{N_j} (-b)^{x-\gamma_j} G_{N_1,N_2}(x)\delta(y-x),
    \quad
    j=1,2,
\end{equation}
are of constant signs. Given~$n\in\mathbb Z_+$, ~$n\le\min\{N_1,N_2\}$, determine the
polynomial~$P_n$ of degree~$2n$ from the formula:
\[
    P_n(x) (-b)^x G_{N_1,N_2}(x)=\nabla^n \left[ (-b)^x G_{N_1-n,N_2-n}(x)\right].
\]
This polynomial satisfies the orthogonality conditions~\eqref{eq:orth_mdiscr2} with respect to
the measures from~\eqref{eq:angkr_measures}. In literature, such families of polynomials are
usually called the Angelesco systems, see~\cite{Angelesco,AptStahl,Kaljagin}. So, we call~$P_n$
the \emph{Angelesco-Kravchuk polynomial}.
\begin{figure}[ht]
    \caption{All zeros of the Angelesco-Kravchuk polynomials are real. This picture shows the
        zeros of polynomials on the real line for each~$n$ when the parameters are:
        \[b=2,\quad N_1=N_2=20,\quad \gamma_1=20,\quad \gamma_2=-\frac 12,\]
        and therefore~$S_{\rho_1}=\{20,21,\dots,40\}$
        and~$S_{\rho_2}=\{-0.5,0.5,1.5,\dots,19.5\}$.}%
    \label{fig:6}    
    \includegraphics[width=16cm]{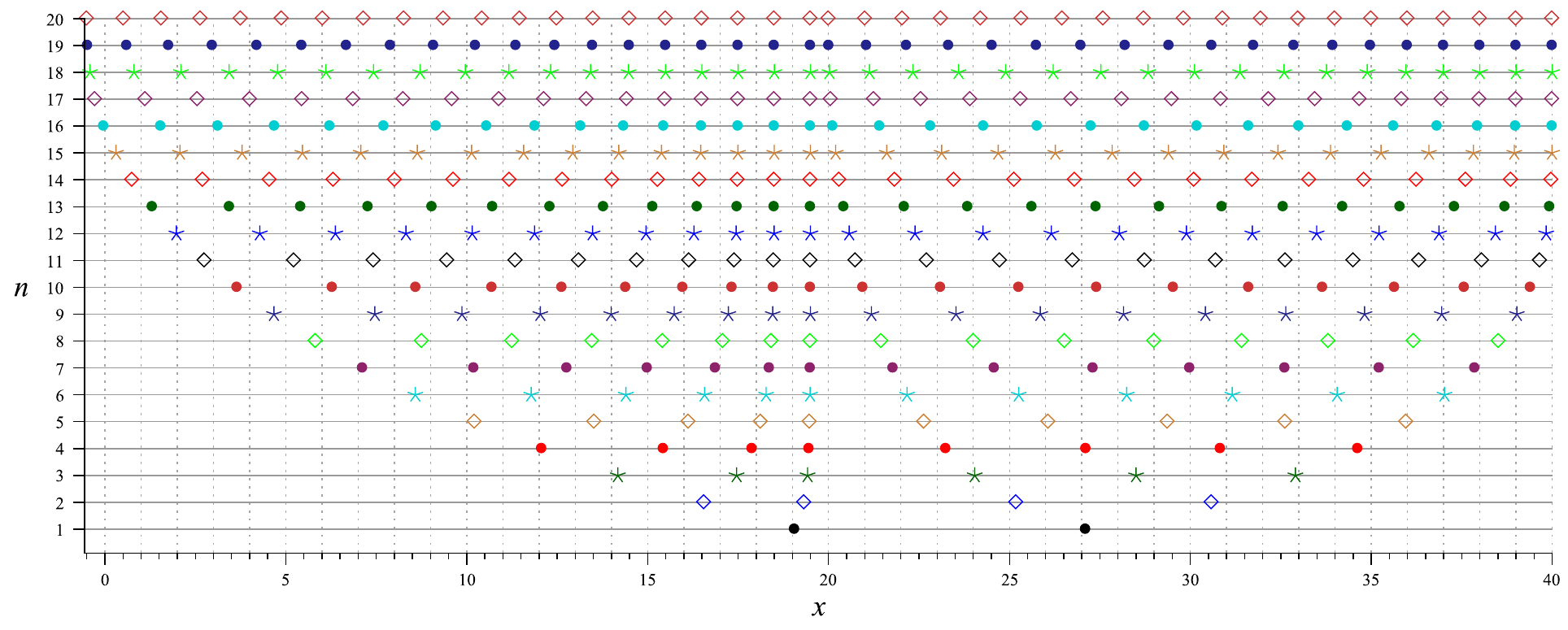}
\end{figure}

\subsubsection*{Kravchuk-Hahn II polynomials}
Now, fix~$B\coloneqq -1$ and let the parameters~$N_1,N_2\in\mathbb N$
and~$\alpha,\beta,\gamma_1,\gamma_2\in\mathbb R$ satisfy one of the following two conditions:
\begin{enumerate}[(a)]\itemsep7pt
\item\label{item:hk2_a}
    $\alpha > \gamma_2-\gamma_1>N_1$, ~$\gamma_2-\gamma_1\notin \mathbb{N}$, ~$\beta>0$;
\item\label{item:hk2_b}
    $-\beta-N_2 > \gamma_2-\gamma_1>N_1$, ~$\gamma_2-\gamma_1\notin \mathbb{N}$, ~$-\alpha-N_2>0$,\\
    $\alpha+\gamma_1-\gamma_2, \notin-\mathbb{Z}_+$,
    ~$\beta-\gamma_1+\gamma_2\notin-\mathbb{Z}_+$, and~$\alpha,\beta \notin-\mathbb{Z}_+$.
\end{enumerate}
Under these conditions, the measures~$\mu_1,\mu_2$ defined in~\eqref{eq:krHn_measures} keep
their signs. Moreover, the convex hulls of~$\supp\mu_1$ and~$\supp\mu_2$ are disjoint.

Given~$n\in\mathbb Z_+$, ~$n\le\min\{N_1,N_2\}$, determine the
\emph{Kravchuk-Hahn II polynomial} as the polynomial~$P_n$ satisfying the
formula~\eqref{eq:krHn_Rodrigues}. The polynomial~$P_n$ has degree~$2n$, it is
orthogonal~\eqref{eq:orth_mdiscr2} with respect to the measures from~\eqref{eq:krHn_measures}.
Here each of~\eqref{item:hk2_a} and~\eqref{item:hk2_b} implies Condition~MD2.

\begin{figure}[ht]
    \caption{All zeros of the Kravchuk-Hahn II polynomials are real. This picture shows the
        zeros of polynomials on the real line for each~$n$ when
        \[\alpha=21,\quad\beta=\frac 43,\quad N_1=N_2=20,\quad \gamma_1=20,\quad \gamma_2=-\frac 13.\]}%
    \label{fig:7}
    \includegraphics[width=16cm]{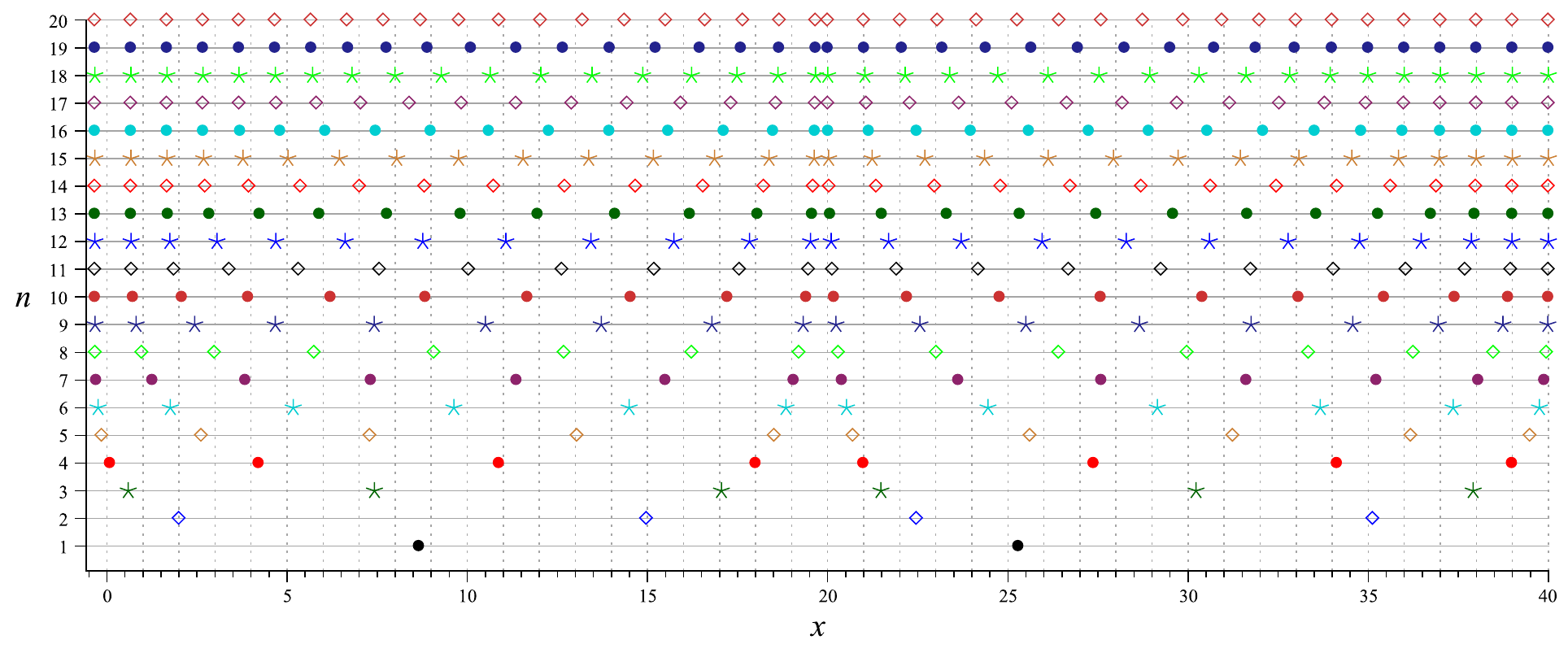}
\end{figure}

\section{Conclusion}

We considered several examples of polynomials of multiple discrete orthogonality for the case
of~$r=2$ weights. For each of the examples, the weight functions satisfy the difference Pearson
equation, and the related polynomials may be found from the Rodrigues
formula~\eqref{eq:md_Rodrigues}.

Clear, that the Rodrigues formula implies recurrence relations of the third order which relate
the polynomials~$P_n$ of degree~$2n$ corresponding to the diagonal indices~$(n,n)$ with the
polynomials~$Q_n$ of degree~$2n+1$ corresponding to the indices~$(n+1,n)$. (The
polynomials~$Q_n$ satisfy~$n+1$ orthogonality conditions with respect to the measure~$\mu_1$
and~$n$ conditions with respect to~$\mu_2$.) It seems interesting to study the properties of
these recurrence relations, as well as the corresponding band Hessenberg 
operators.

The present work does not touch upon the problems of \emph{normality} of indices, as well as the
uniqueness of multiple orthogonal polynomials~$P_n$. For the case~II, when the convex hulls of
the supports of the measures~$\mu_1,\mu_2$ are disjoint, the system of the
measures~$\mu_1,\mu_2$ is \emph{perfect}. It means that all indices~$(n_1,n_2)$ are normal, and
the corresponding multiple orthogonal polynomial is determined uniquely up to a normalizing
constant. Indeed, the orthogonality relations yield that~$P_{n_1,n_2}$ has~$n_j$ simple zeros in
a convex hull of the support of~$\mu_j$, see Figures~\ref{fig:6}--\ref{fig:7}. That is, the
degree of~$P_{n_1,n_2}$ is equal to~$n_1+n_2$, which gives the normality of~$(n_1,n_2)$. For the
case~I, the question of normality requires further study.

Classical orthogonal polynomials are limits of classical polynomials of discrete orthogonality.
Various limiting relations between hypergeometric orthogonal polynomials are represented in the
Askey scheme~\cite{KoeSw98}. The multiple orthogonal polynomials considered here possess
analogous properties.

Among interesting open questions is the asymptotic behaviour of~$P_n$ as~$n\to\infty$. The
paper~\cite{Sor2010} studied the asymptotics of the scaled Meixner-Sorokin polynomials (case I).
It gave rise to new non-trivial equilibrium problems for the vector logarithmic potential. Part
of zeros of~$P_n$ sweep out curves in the complex plane. The case~II with appropriate scaling is
awaited to yield standard equilibrium problems with the Angelesco interaction matrix and an
upper constraint for the equilibrium measure.

We plan to consider the highlighted problems in further publications. We also hope that the
polynomials introduced here will find their applications in combinatorics, representation theory
and other areas.

\end{document}